\documentclass[11pt]{article}
\usepackage{lineno}
\usepackage[usenames]{color}
\usepackage{amsmath}
\usepackage{amsthm} 
\usepackage{amsfonts} 
\usepackage{amssymb}
\usepackage{graphicx} 
 
\definecolor{LightGray}{rgb}{0.7,0.7,0.7}

\newtheorem{theorem}{Theorem}

\newtheorem{definition}{Definition}
 
\newcommand{\figref}[1]{Figure~\ref{#1}}		
\newcommand{\thmref}[1]{Theorem~\ref{#1}}

\newcommand{\st}[1]{\, |\, \text{#1} }	
\newcommand{\set}[1]{\left\{#1\right\}}

\newcommand{\cS}{{\cal S}}

\newcommand{\Ga}{\Gamma}

\newcommand{\ka}{\kappa}

\newcommand{\Rgf}[2]{R^{(#1,#2)}}	
\newcommand{\Lgf}[2]{L^{(#1,#2)}}	
\newcommand{\reslukapart}[2]{Z^{(#1,#2)}}	

\renewcommand{\H}[2]{H^{(#1,#2)}}	
\renewcommand{\c}[2]{c^{(#1,#2)}}	

\newcommand{\crit}{A}	

\DeclareMathOperator{\result}{Res}

\DeclareMathOperator{\rightViz}{RightViz}	

\begin{document} 
\title{An infinite family of adsorption models and restricted Lukasiewicz paths}

\author{R.~Brak$^1$, G.K~Iliev$^1$, T.~Prellberg$^2$,\\
$^1$Department of Mathematics and Statistics,\\
	The University of Melbourne,\\
	Parkville, Victoria 3010, Australia\\
	$^2$ School of Mathematical Sciences,\\ 
	Queen Mary University of London,\\
	Mile End Road, London E1 4NS, UK.
}
\date{\today}
\maketitle	
\begin{abstract}

We define $(k,\ell)$-restricted Lukasiewicz paths, $k\le\ell\in\mathbb{N}_0$, and use
these paths as models of polymer adsorption.  We write down a polynomial expression 
satisfied by the generating function for arbitrary values of $(k,\ell)$.  The resulting 
polynomial is of degree $\ell+1$ and hence cannot be solved explicitly for sufficiently 
large $\ell$.  We provide two different approaches to obtain the phase diagram. In addition
to a more conventional analysis, we also develop a new mathematical characterization of the 
phase diagram in terms of the discriminant of the polynomial and a zero of its highest 
degree coefficient.

We then give a bijection between $(k,\ell)$-restricted Lukasiewicz paths and 
``rise''-restricted Dyck paths, identifying another family of path models which 
share the same critical behaviour. For $(k,\ell)=(1,\infty)$ we provide a 
new bijection to Motzkin paths.

We also consider the area-weighted generating function and show that it is 
a $q$-deformed algebraic function. We determine the generating function 
explicitly in particular cases of $(k,\ell)$-restricted Lukasiewicz paths, and
for $(k,\ell)=(0,\infty)$ we provide a bijection to Dyck paths.

\end{abstract}

\vfill
\textsc{Short Title:}  An infinite family of adsorption models.



\newpage	

\begin{center}
This paper is dedicated to Cyril Domb on the occasion of his 90th birthday.
\end{center} 

\section{Introduction and definitions} 

The study of the statistical mechanics of polymers has been a topic of much interest for
 nearly 70 years, with a great deal of focus devoted to systems of long, linear molecules in 
a good solvent  \cite{dimarzio:1971ev,rubin:1965ve}.  The class of simple models known as directed paths
have received much attention when studying the behaviour of such molecules in the presence of an impenetrable surface  \cite{forgacs:1991fj,orlandini:2004kx,whittington:1998lr}.

In this paper we propose a new discrete  two-parameter family of directed path models, 
\emph{$(k,\ell)$-restricted Lukasiewicz paths}, with $k\le\ell\in \mathbb{N}_0$, that have a 
tuneable step set.
This family of paths is a generalization of well-known directed path models \cite{brak:2007xe,prellberg:1995xg,prellberg94c} and
for some particular choices of the parameters leads to several interesting bijections to 
classical directed path models.

In addition to defining the paths, we use them to study the problem of adsorbing polymers 
at an impenetrable surface.  The generating function for each model satisfies a polynomial equation
of degree $\ell$+1, and hence cannot be solved explicitly for large $\ell$.  We present a new method
of extracting the phase diagram from these polynomials.  The phase diagram for each member of the
family is composed of two regimes: 1) a regime independent of the 
contact parameter and 2) a regime which depends explicitly on the contact parameter.  
We identify the first regime by evaluating a discriminant, while the second regime is 
the physically relevant solution of a high-degree algebraic equation.  We use \emph{PGL$(2)$-invariance}  
of the discriminant to show that  Regime 1 is indeed independent of the contact parameter for
all values of $(k,\ell)$. 

We also give a weight preserving bijection from the  $(k,\ell)$-restricted Lukasiewicz paths to $(k+1,\ell+1)$-\emph{rise restricted Dyck paths} and hence give another family of models which have the same phase digram as the Lukasiewicz models. 
The set of all these models is conveniently illustrated in \figref{fig_twoparam} which shows the location of Dyck and Motzkin paths and of  two additional bijections discussed in this paper.  

\begin{figure}[htb] 
	\centering
 		\includegraphics[scale=1.2]{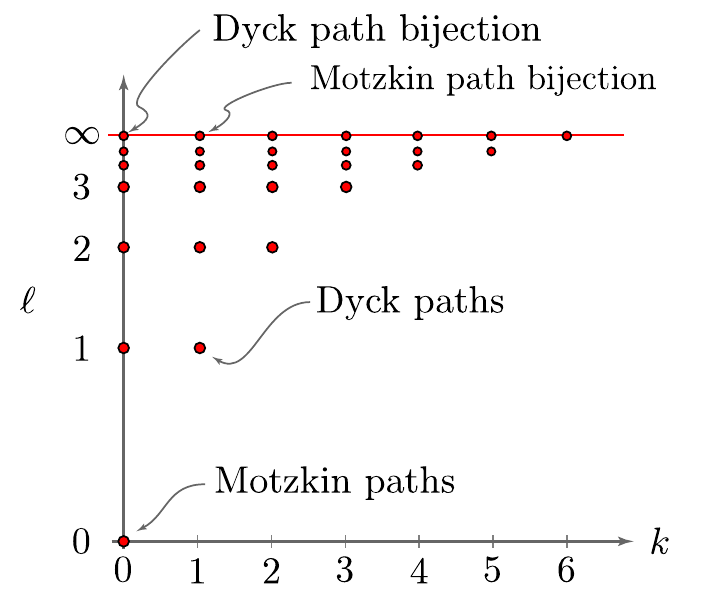}
	\caption{An illustration of the discrete two parameter family of adsorption models given by $(k,\ell)$-restricted Lukasiewicz paths. The points $(1,1)$ and $(0,1)$ correspond to Dyck paths and Motzkin paths, respectively. At the points $(0,\infty)$ and $(1,\infty)$ there are also bijections to Dyck and Motzkin paths, as indicated in the figure.
}
	\label{fig_twoparam}
\end{figure}

Finally, we find an equation satisfied by the area-weighted generating function  of $(k,\ell)$-restricted Lukasiewicz paths which we solve in two cases, namely $(k,k)$ and $(0,\infty)$. The latter solution can also be obtained via a bijection to Dyck paths which we present. The area-weighted generating function gives a simple model of a single membrane vesicle above a surface with adsorption. The area generating variable corresponds to a   `volume' fugacity.

\begin{definition}\label{def_paths}
Let $\mathbb{N}_0=\set{0,1,2,3,\cdots}$. 
A \textbf{length} $n$ \textbf{direct path} is a sequence of vertices $ v_0 v_1\dots v_{n}$ with  $v_i=(x_i,y_i)\in \mathbb{N}_0\times \mathbb{N}_0$, $v_0=(0,0)$ and  $v_{n}=(n,0)$,
where the \textbf{steps}, $v_i-v_{i-1}$, belong to a given \textbf{step set} $\cS\subseteq\{1\}\times\mathbb{Z}$. Choosing 
\[
	\cS=\set{(1,-1)}\cup\set{(1,j)\st{  $k\le j\le \ell$, and $k,\ell\in  \mathbb{N}_0$}   }
\] 
defines \textbf{$(k,\ell)$-restricted Lukasiewicz paths}.
A step $(1,j)$ is called a \textbf{jump $j$ step}.The \textbf{height} of a vertex $v_i=(x_i,y_i)$ is  $y_i$ and the height of a step is the height of its first (\emph{i.e.}\ left) vertex.  A \textbf{contact} weight, $a$, is associated with any vertex $v_1,\ldots v_n$ of height zero.
\end{definition}	
Note, contact weights are only associated with returns to the surface \emph{i.e.}\ $v_0$ does not contribute a contact weight.

Clearly Dyck paths are $(1,1)$-restricted Lukasiewicz paths, while Motzkin paths are $(0,1)$-restricted Lukasiewicz path. An example of a $(1,2)$-restricted Lukasiewicz path is shown in Figure \ref{fig:figs_riseEx}.
	
\begin{figure}[htb]
	\centering
		\includegraphics[width=0.8\textwidth]{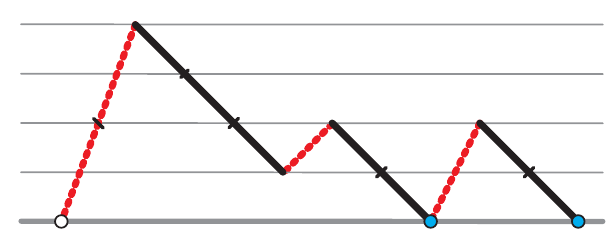}
	\caption{An example of a $(1,2)$-restricted Lukasiewicz path of length eleven with two contacts. The jump  steps are shown in colour (dashed).
	}  
	\label{fig:figs_riseEx}
\end{figure}

\section{Contact polynomials for Restricted Lukasiewicz paths} 
\label{sec:contact_polynomials_for_restricted_lukasiewicz_paths}

We now consider the partition functions and their associated generating functions. Let $\Rgf{k}{\ell}(z;a)$ be the generating function for $(k,\ell)$-restricted Lukasiewicz paths    with partition functions $\reslukapart{k}{\ell}_n(a)$, namely,
\begin{equation}
	\Rgf{k}{\ell}(z;a)=\sum_{n\ge0} \reslukapart{k}{\ell}_n(a) z^n.
\end{equation}
The following theorem gives the algebraic equation satisfied by $\Rgf{k}{\ell}(z;a)$ in terms of $\Lgf{k}{\ell}(z)=\Rgf{k}{\ell}(z;1)$.
\begin{theorem}\label{thm_lukaeqs}

The generating function $\Rgf{k}{\ell}(z;a)$ is given by the following pair of algebraic equations
\begin{align}
	\Rgf{k}{\ell}(z;a)=1+ az  \sum_{j=k}^\ell \left(z \Lgf{k}{\ell}(z)\right)^j \, \Rgf{k}{\ell}(z;a)\label{eq_rgf}\\
	\Lgf{k}{\ell}(z)=1+ \sum_{j=k}^{\ell} \left(z \Lgf{k}{\ell}(z)\right)^{j+1}\label{eq_lgf}\;.
\end{align}
\end{theorem}
Since the derivation of these equations is a direct generalization of known methods \cite{brak:2007xe,prellberg:1995xg,prellberg94c} we provide only an outline of
the proof.

\begin{proof}[Outline]\label{pf_lk1}

We note that these equations arise by partitioning the set of all Lukasiewicz
paths weighted by contacts according to the height of the leftmost jump step. After jumping to height
$j$ the path must take $j$ down steps in order to return to the surface. 
Between each of these down steps  a sequence of steps  all above a fixed height,  with corresponding generating function $\Lgf{k}{\ell}(z)$,  are permitted.  
These sequences are denoted schematically in the  \figref{fig_lukafact} by a `loop'.
After returning to the surface the paths can finish with the corresponding Lukasiewicz path. This     
factorization argument is illustrated schematically in \figref{fig_lukafact}.
Note that the jump step gets a weight $z$ irrespective of the jump height.
\end{proof}

\begin{figure}[htb]
	\centering
 		\includegraphics[scale=1.2]{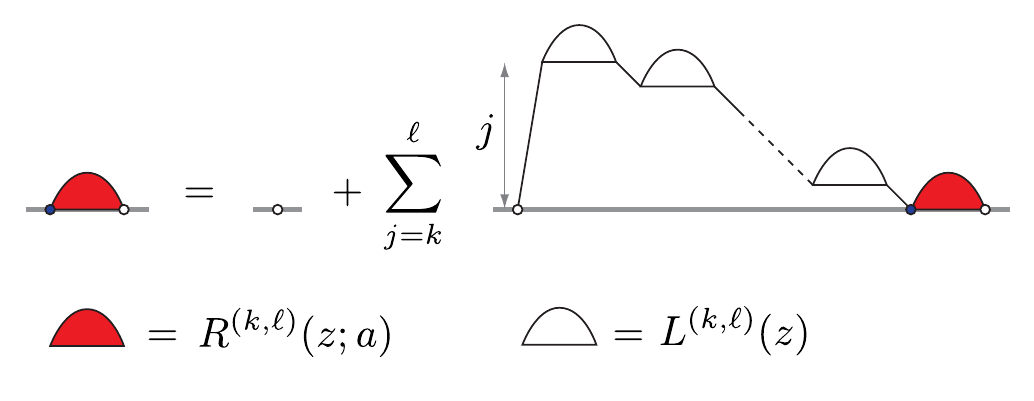}
	\caption{Schematic representation of the   Lukasiewicz path factorisations.}
	\label{fig_lukafact}
\end{figure}

We remark that by solving for the summation in \eqref{eq_lgf} and using it in \eqref{eq_rgf}, it follows that 
$\Rgf{k}{\ell}(z;a)$ and  $\Lgf{k}{\ell}(z)$ are related   by the simple equation
\begin{equation}
\Lgf{k}{\ell}(z) =\frac{a 	\Rgf{k}{\ell}(z;a)}{1+(a-1)	\Rgf{k}{\ell}(z;a) } \;.\label{simple2}
\end{equation}
Substituting \eqref{simple2} into \eqref{eq_rgf} shows that $	\Rgf{k}{\ell}(z;a)$ satisfies a degree $\ell+1$ polynomial.

\subsection{Singular Behaviour and Free Energy} 
The free energy $\ka_{k,\ell}(a)$ for each model is related to the radius of convergence $z_c(a)$ of  the generating function $\Rgf{k}{\ell}(z)$ by,
\[
	\ka_{k,\ell}(a)=-\log z_c(a)\;.
\] 
We deduce the radius of convergence of $\Rgf{k}{\ell}(z)$ in two ways, i) from the singular behaviour of  $\Lgf{k}{\ell}(z)$ and its
relationship to $\Rgf{k}{\ell}(z)$, namely equation \eqref{simple2}, and ii) a more direct approach which gives the 
singularity structure of $\Rgf{k}{\ell}(z)$  without first finding that of $\Lgf{k}{\ell}(z)$.  The first method gives the phase boundary of each 
model in terms of a  unique positive root of a certain polynomial. The second approach yields a more general result that the high 
temperature regime of the phase boundary arises from the root of the ``indicial'' equation and the low temperature regime  from the discriminant 
of the polynomial satisfied by $\Rgf{k}{\ell}(z)$ -- see \figref{fig_genfe} below.  

\subsubsection{Radius of convergence of $\Rgf{k}{\ell}$ from $\Lgf{k}{\ell}(z)$.}

The following theorem states the radius of convergence of $\Rgf{k}{\ell}(z)$, and  hence the phase boundary of the models, all in terms of the unique positive root of a certain degree $\ell$ polynomial. The proof uses the singular behaviour of  $\Lgf{k}{\ell}(z)$ and its rational relation to $\Rgf{k}{\ell}(z)$.  

\begin{theorem}\label{thm_sing} 
Let $u_c$ be the unique positive real root of the polynomial
\begin{equation}
		 \sum_{j=k}^\ell j u^{j+1}=1
\end{equation}
and
\begin{align}
	 z_c=&\frac{u_c}{1+\Gamma_{k,\ell}(u_c)}\;,\\
	a_c=&1+\frac{1}{\Gamma_{k,\ell}(u_c)}\;,
	\end{align}
	where 
	\begin{equation}\label{eq_gampoly}
			\Gamma_{k,\ell}(u) =\sum_{j=k}^\ell u^{j+1}\;.
	\end{equation}
Then the  radius of convergence  $z_c(a)$  of  $\Rgf{k}{\ell}(z;a)$ for $a\ge1$ is  
\begin{equation}
		z_c(a)=\begin{cases}
			z_c & \text{if $1\le a\le a_c$}\\
		z^+_c(a)	& \text{if $a>a_c$}
		\end{cases}
\end{equation}
where  $z^+_c(a)$ is the unique positive real root of 
\begin{equation}\label{eq_zplus}
	\frac1{a-1} = \Gamma_{k,\ell}\left(\frac{az}{a-1}\right)\;.
\end{equation}
\end{theorem}

\begin{proof}\label{pf:}

To reduce the notational clutter in the proof we use
\begin{align*}
	R=\Rgf{k}{\ell}(z;a)\\
	L=\Lgf{k}{\ell}(z).
\end{align*}
From \thmref{thm_lukaeqs} we have   
\begin{equation}
\label{L}
L=1+\sum_{j=k}^\ell(zL)^{j+1} = 1+\Gamma_{k,\ell}(zL)
\end{equation}
and
\begin{equation}
\label{LR}
L=\frac{aR}{1+(a-1)R}.
\end{equation}
Since $L(z)$ is a generating function (with positive coefficients), for $z\ge0$ it is an increasing function of $z$. From \eqref{L} it follows that $L(0)=1$ and $L(z)$ has a singularity at $z_c \leq 1$ on the positive real axis. 

From \eqref{L} we compute the derivative as
\begin{equation}
\label{Lprime}
L'=\frac{L\sum\limits_{j=k}^\ell(j+1)(zL)^j }{1-z\sum\limits_{j=k}^\ell(j+1)(zL)^j }\;.
\end{equation}
At $z_c$, $L'$ must diverge, i.e. the denominator vanishes. This implies that at $z_c$
\begin{equation}
\label{Lprimeinf}
z_c \sum_{j=k}^\ell(j+1)(z_cL)^j =1\;.
\end{equation}
Therefore $L$ has a finite value $L_c$ at $z_c$ (unless $k=\ell=0$). 

Combining (\ref{L}) and (\ref{Lprimeinf}) gives
\begin{equation}
\label{u}
\sum_{j=k}^\ell ju^{j+1}=1\;.
\end{equation}
with $u=zL$, which clearly  has a unique positive solution, denoted  $u_c$.
We can now express both $z_c$ and $L_c$ in terms of $u_c$ as
\begin{equation}
\label{Lcrit}
L_c=1+\Gamma_{k,\ell}(u_c)
\end{equation}
and
\begin{equation}
\label{zcrit}
z_c=\frac{u_c}{1+\Gamma_{k,\ell}(u_c)}\;.
\end{equation}

We now turn our attention to the singularity structure of $R$. Inverting (\ref{LR}) gives
\begin{equation}
\label{RL}
R=\frac{L}{a-(a-1)L}\;.
\end{equation}
Clearly $R$ is singular at $z_c$. This singularity can be dominated by a second singularity $z^+$ arising from a vanishing denominator, i.e.
if
\begin{equation}
\label{asing}
L=\frac a{a-1}
\end{equation}
and thus, using \eqref{L}, we get that $z^+$ is given by the unique positive root of
\begin{equation}\label{eq_indeq}
\frac 1{a-1} = 1+ \Gamma_{k,\ell}\left(\frac {az}{a-1}\right)	\;.
\end{equation}
Rearranging the above equation gives \eqref{eq_zplus}. By choosing $a$ arbitrarily large, \eqref{eq_indeq} implies that $z^+$ can be arbitrarily small. If $a$ is sufficiently large this will be the the closest singularity to the origin, \emph{i.e.}\ $0<z^+\le z_c$.   At $z_c$ both singularities 
coincide, and the critical value $a_c$ is determined by
\begin{equation}\label{eq_aceq}
a_c=1+\frac{1}{\Gamma_{k,\ell}(u_c)}\;.
\end{equation}
\end{proof}

As an example, a plot of $a_c $ is shown for   $k=1$ and $\ell=1,2,3,\dots,8$   in \figref{fig_acseries}. 

\begin{figure}[htb] 
	\centering
 		\includegraphics[scale=1.2]{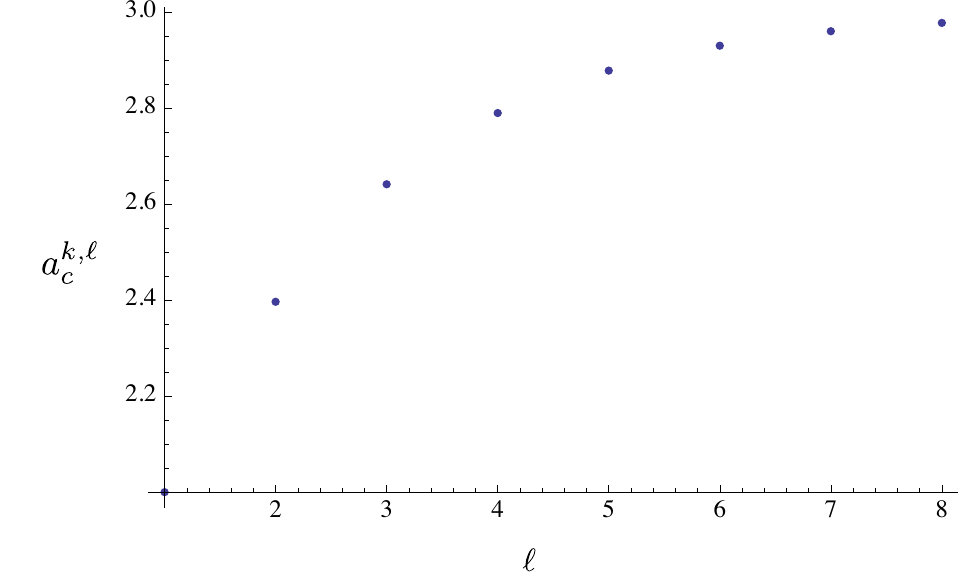}
	\caption{A plot of $a_c^{(k,\ell)}$   for   $k=1$ and $\ell=1,2,3,\dots,8$. The point $(1,1)$ is Dyck paths, where $a_c=2$, and the limiting value $\left.{a_c}\right|_{k=1,\ell=\infty}=3$ is the Motzkin path bijection point. The latter bijection  point has  non-standard surface weights (\emph{i.e.}\ the bijection is not weight preserving) and hence \mbox{$a_c\ne 3/2$}.
}
	\label{fig_acseries}
\end{figure}

\subsubsection{Singular behaviour of $\Rgf{k}{\ell}$ from its algebraic structure.}
\label{sec:singular_behavious}

In this section we focus on the fact that the generating function $\Rgf{k}{\ell}(z)$ is an algebraic function, \emph{i.e.}\ it is the root of   a degree $\ell+1$ polynomial  of the form
\begin{equation}\label{eq_Rpoly}
	p_{k,\ell+1}(z,a) R^{\ell+1}+p_{k,\ell}(z,a) R^{\ell}+\dots+p_{k,0}(z,a)=0.
\end{equation} 
 We show that in one regime the free energy arises from the 
discriminant of \eqref{eq_Rpoly},   whilst in another regime the free energy arises from the zeros of $p_{k,\ell+1}$. This generic structure is shown schematically in  \figref{fig_genfe}.
\begin{figure}[htb] 
	\centering
 		\includegraphics[scale=1.2]{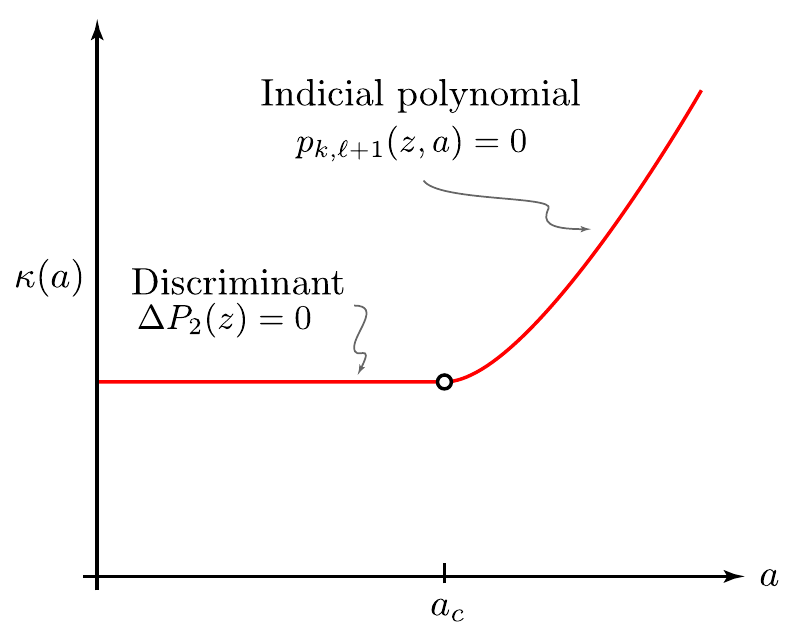}
	\caption{Adsorption free energy as determined by the discriminant and indicial equation.}
	\label{fig_genfe}
\end{figure}

Since $R=\Rgf{k}{\ell}(z;a)$ satisfies a degree $\ell+1$ algebraic equation,
there are at most $\ell+1$ solutions, $ \Rgf{k}{\ell}_i(z;a)$, one of which will correspond to the generating function.    
For finite $z$ and $a$ considered  a 
parameter there are only two sources of non-analyticity in  $\Rgf{k}{\ell}(z;a)$ -- see Theorem 12.2.1 of Hille \cite{hille:1972qy}. Either i)   
$\Rgf{k}{\ell}(z;a)\to\pm\infty$ as $z\to z_c(a)$ or ii) the branch structure of  $\Rgf{k}{\ell}(z;a)$ changes at 
$z_c(a)$. 

The non-analytic  points $z_c(a)$ in case i) arises from the zeros of 
$p_{k,\ell+1}(z,a)$, which we will call the \textbf{indicial} equation\footnote{This coefficient does not appear to have a standard name. Thus we name it analogous to that from differential equations.}. 
On the other hand, the non-analytic    points in case ii) arise from the zeros of the discriminant of \eqref{eq_Rpoly}. If $z_c(a)$ is real then the zeros of the discriminant occur where the curve $\Rgf{k}{\ell}_i(z;a)$ is finite but  has infinite slope.

An explicit expression for the indicial equation for $(k,\ell)$-restricted Lukasiewicz paths is readily obtained by 
substituting \eqref{simple2} into \eqref{eq_rgf} and extracting the coefficient of  $ \Rgf{k}{\ell}(z;a)^{\ell+1}$, which gives
\begin{align}
	p_{k,\ell+1}(z,a) = & (a-1)^\ell-\sum_{j=k}^\ell (az)^{j+1} (a-1)^{\ell-j}\\
	=& (a-1)^\ell-  (a-1)^{\ell-1}\, \Gamma_{k,\ell}\left(\frac{a z}{a-1}\right)\;,
\end{align}
where $ \Gamma_{k,\ell}$ is given by \eqref{eq_gampoly}.
Note, this is the same equation as \eqref{eq_indeq} -- thus we see the zero in the denominator of \eqref{RL} is the same as the (unique positive) root of the indicial equation. Following the same argument as in the proof of \thmref{thm_sing} is follows that for $a$ sufficiently large this unique root must be the   radius of convergence of $R$ since, as show below,  the  discriminant zeros are $a$ independent. 

The discriminant, denoted $\Delta P$, of a polynomial $P(R)$ whose coefficients are polynomials in $z$  is the resultant of $P(R)$ and its derivative 
\[
	\Delta  P(z)=\result\left(P,\frac{\partial P}{\partial R};R\right)\;,
\] 
and thus may be obtained from a Sylvester determinant \cite{moiseevich:1994qy}. Below we will prove that for $(k,\ell)$-restricted Lukasiewicz paths the discriminant always possesses  a factor in $z$ independent of $a$. 

In order to prove this generic structure we need the discriminant of the polynomial satisfied by $R=\Rgf{k}{\ell}(z;a)$, that is
\begin{equation}\label{eq_Rpoly2}
	P_1=\,p_{k,\ell+1}(z,a) R^{\ell+1}+p_{k,\ell}(z,a) R^{\ell}+\dots+p_{k,0}(z,a).
\end{equation}
We want to express the discriminant of \eqref{eq_Rpoly2} in terms of the 
discriminant of the polynomial satisfied by $L=\Lgf{k}{\ell}(z)$,
obtained from \eqref{eq_lgf} as 
\[
P_2=\	\sum_{j=k}^{\ell} \left(z L\right)^{j+1}-L+1\;.
	\label{eq_lgf2}
\]
We are aided by the following theorem (for a proof see \cite{moiseevich:1994qy}).
\begin{theorem}(PGL(2)-invariance)\label{thm_flatfree}
Let $Q(z)$ be a polynomial of degree $n$ and $\displaystyle r: z \mapsto \frac{\alpha z+\beta}{\gamma z+\delta}$ 
with $\alpha,\beta,\gamma,\delta \in \mathbb{R}$.  The discriminant of 
\[
	(\gamma z+\delta )^nQ\bigl(r(z)\bigr)
\] is given by
\[
	(\alpha \delta - \beta \gamma)^{n(n-1)} \Delta Q\;,
\]
where $\Delta Q$ is the discriminant of $Q(z)$.
\end{theorem}
Note, the $(\gamma z+\delta )^n$ factor is present to clear the denominator introduced by the substitution $r$.
We now apply \thmref{thm_flatfree} to the polynomial   in  \eqref{eq_Rpoly2} along with the inverse of \eqref{simple2}, namely,
\begin{equation}
  \Rgf{k}{\ell}(z;a)=\frac{\Lgf{k}{\ell}(z)}{a-(a-1)\Lgf{k}{\ell}(z)}.
\label{simple3}
\end{equation}
The result is the discriminant  of polynomial $P_1$   given in terms   of the `contact-free' polynomial, $P_2$, that is
\begin{equation}
	\Delta P_1 (z;a)=a^{\ell(\ell+1)}\, \Delta P_2(z)\;.
	\label{eqdiscP1}
\end{equation}
This shows that the $a$ dependence of the discriminant of $P_1$ is contained in the $a^{\ell(\ell+1)}$ factor,
and therefore that the roots (as a polynomial in $z$) of $ \Delta P_1 (z;a)$ are independent of $a$.

We identify the critical contact weight, $a_c$,  as the value of $a$ for which the discriminant \eqref{eqdiscP1} and indicial polynomials have simultaneous zeros,
\begin{align}
	\Delta P_2(z_c)&=0\\
		p_{k,\ell+1}(z_c,a_c)&=0.
\end{align}
Eliminating $z_c$ from these equations by taking the resultant gives us a polynomial equation satisfied by $a_c$, namely,
\begin{equation}
\crit_{k,\ell}(a)=	\result\left(\Delta P_2(z),p_{\ell+1}(z,a);z\right).
\end{equation}
This is the same equation that would be obtained by eliminating $u_c$ between \eqref{eq_aceq} and  \eqref{u} by taking the resultant.

For example,  for  Motzkin paths it is simple to obtain
\[
\crit_{0,1}(a)=a^4 (2 a-3)^2
\]
which gives the familiar result $a_c=3/2$. However, for other values of $k$ and $\ell$ the equation can become rather complicated, for example,  
\[
\crit_{2,4}(a)=7 a^5-113 a^4+770 a^3-2756 a^2+5180
	   a-4112
\]
which does not factor over $\mathbb{Z}$; as such the root has to be found numerically.

\section{Rise restricted Dyck path bijection} 
\label{sec:bijections}

We now show that the $(k,\ell)$-restricted Lukasiewicz paths are in bijection with the  
$(k,\ell)$-rise restricted Dyck paths, hence giving another family of models with the same critical behaviour.
\begin{definition}[Rise, valley, peak, hook]\label{def_rise}
	Let $s_1s_2\dots s_n$ be the step sequence of a Dyck path. A \textbf{valley} (resp.\ \textbf{peak}) is a pair $s_is_{i+1}$ with $s_i$ a down (resp.\ up) step and $s_{i+1}$ an up (resp.\ down) step. A rise of \textbf{length} $j$ is a maximal subsequence of $j$ steps,  $r_{i,j}=s_{i}s_{i+1}\dots s_{i+j-1}$ such that 
	\begin{itemize}
		\item $r_{i,j}$ contains no valleys or peaks
		\item $s_{i-1}s_{i}$ is a valley (or $i=1$) and  $s_{i+j-1}s_{i+j}$ is a peak.
	\end{itemize}
A \textbf{hook} is a rise and the down step of the peak, that is the subsequence $r_{i,j} s_{i+j}$. The length of the hook is $j+1$. A  Dyck path is \textbf{$(k,\ell)$-rise restricted} iff the length $j$ of  all rises satisfies $k\le j\le \ell$.
\end{definition}
Thus a $(k,\ell)$-rise restricted Dyck path contains no rises shorter than $k$ or greater than $\ell$. An example of a Dyck path with rises shown is given in \figref{fig_rise}.
\begin{figure}[htb] 
	\centering
 		\includegraphics[scale=1.2]{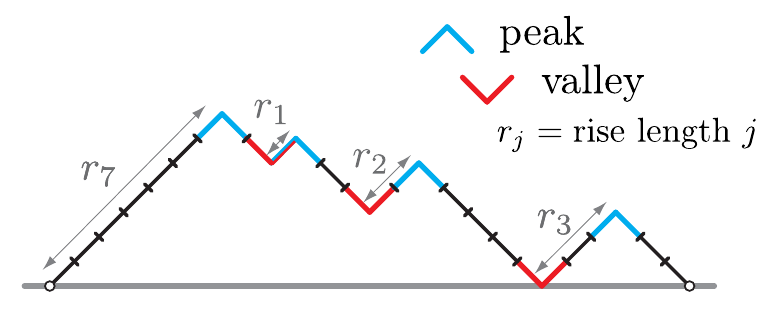}
	\caption{An example of a Dyck path shown the rises, hook, peaks and valleys.}
	\label{fig_rise}
\end{figure}

We can now state the following bijection.
\begin{theorem}\label{thm_luka_rise}
	The set of $(k,\ell)$-restricted Lukasiewicz paths of length $n$  is in bijection with the set of  $(k+1,\ell+1)$-rise restricted Dyck paths of length $2n$. Furthermore the bijection preserves the contact weight of the path.
\end{theorem}

The bijection is a generalisation of the classical bijection between Lukasiewicz paths (with no jump restriction) of length $n$ and Dyck paths of length $2n$. The idea of the classical bijection is to replace the jump $j$ step of the Lukasiewicz path with a length $j+2$ hook as illustrated below.
\begin{center}
		\includegraphics[scale=1.2]{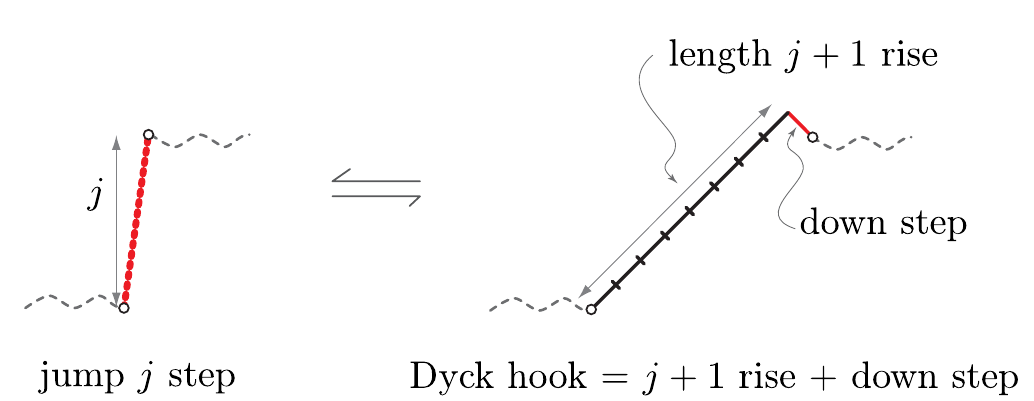} 
\end{center}
Since this is a straightforward generalisation of the  classical bijection we do not provide a detailed proof, only the outline.  

\begin{proof}[Proof outline]\label{pf_lbij}
First note that the hook-to-rise replacement does not change the	height of any of the Lukasiewicz vertices (only their $x$-coordinates) and hence all the steps after the replacements are above the surface (\emph{i.e.}\ it must be a Dyck path). After the hook-to-rise replacement, two consecutive Lukasiewicz jump steps are separated by a peak (at the end of the first hook)  in the Dyck path, and hence going back from a Dyck path to a Lukasiewicz path, the heights of the jump steps are well defined. 

All that remains to   show is that the resulting Dyck path is length $2n$  \emph{i.e.}\ has twice as many steps as the Lukasiewicz path. This follows if we can partition the set of steps of the Dyck path into two sets of equal size and have the rise-to-hook replacement remove one of the two sets. 
The partition is simple; an up step set, $S_u$ and a down step set, $S_d$. We define a bijection $\Gamma: S_d\to S_u$ and show the rise-to-hook replacement removes all the steps in $S_u$.  If $s_i$ is a down step then $s_i'=\Gamma(s_i)\in S_u$ is defined to be the step ``horizontally visible'',  to the left as illustrated in \figref{fig_hviz}.
\begin{figure}[htb] 
	\centering
		\includegraphics[scale=1.2]{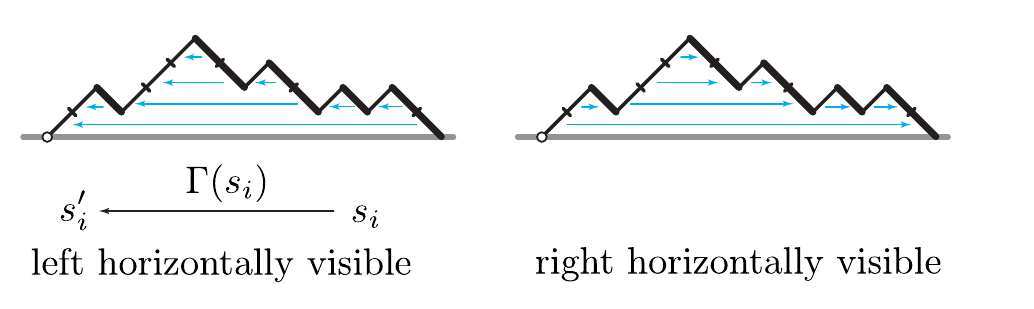}
	\caption{Horizontally visible steps - to the left (left) or to the right (right).}
	\label{fig_hviz}
\end{figure}
 
More precisely $s_i'$ is the rightmost of the set of up steps to the left of $s_i$ that are at the same height as $s_i$. Consider the  rise-to-hook replacement as the composition of two subsequent constructions: i) first delete all the steps in $S_u$ as they constitute the rises (leaving the $n$ down steps of $S_d$) and then ii) replace  the `down' step  of each of the hooks by a jump $j$ \emph{up} step (no change in number of steps).

Clearly if the jump steps are $(k,\ell)$-restricted  then the length of the rises, $j$ are restricted to $k+1\le j\le \ell+1$ as a jump $j$ step maps to a length $j+1$ rise. 

Finally, since the height of the Lukasiewicz path vertices (or inversely the Dyck vertices) are not changed under the bijection the contact weight of the path is unchanged.
\end{proof} 

Since the above bijection is weight preserving, the partition functions for  the Lukasiewicz paths and bijected rise restricted Dyck  paths will be  identical; hence they will have the same thermodynamic properties and thus critical behaviour and phase diagrams.

\section{$(1,\infty)$-restricted Lukasiewicz paths and Motzkin paths} 
\label{sec:plateaux_restricted_motzkin_path_bijection}

\newcommand{\lnp}{L_n^+}	
\newcommand{\gam}{\Ga^-_n}	 
\newcommand{\gao}{\Ga^o_n}

We now consider a new    Motzkin path bijection -- the $(k,\ell)$ location of the corresponding model is shown in \figref{fig_twoparam}. 
The  bijection is somewhat unusual for lattice paths in that it bijects paths of two consecutive lengths  to a single length path. Unfortunately it does \emph{not} preserve the contact weight. However, the generating function is readily obtained (it is a quadratic algebraic function) and the critical contact weight is given by
\[
	a_c=3\;,
\]
which is the limiting point of the sequence shown in \figref{fig_acseries}.

\begin{theorem}\label{thm:plat} Let $M_n$ be the set of Motzkin paths of length $n$ and $L_n$ the set of $(1,\infty)$-restricted Lukasiewicz paths of length $n$. Then there exists a bijection between $\lnp=L_n\cup L_{n+1}$ and $M_n$.
\end{theorem}

\begin{proof}\label{pf:lm}
	Let $l\in \lnp$ and let $l$ have the step sequence $s_1s_2\dots s_k$ (with $k\in\set{n,n+1}$. 	If $s_i$ is a jump $j$ step then denote the set of down steps  which are  `horizontally visible' to the right from $s_i$   as $\rightViz(s_i)$, see \figref{fig_hviz}.
	
	We define two maps $\gao$ and $\gam$ which act on jump steps $s_i$ and the jump steps' associated sequence of horizontally visible down steps, $\rightViz(s_i)$.  Combining these two maps gives the bijection  $\Ga_n: \lnp\to M_n$. We then show that $\Ga_n$ is well defined, injective and surjective and hence a bijection.
The two maps are   defined as follows.

\begin{itemize}
	\item $\gao$: Let $s_i$ be a jump step of $l$ 
	\begin{itemize} 
		\item Replace the jump step by an  up step 
		\item Replace all but the last step  of $\rightViz(s_i)$ by horizontal steps
	\end{itemize}
	\item $\gam$: Let $s_i$ be a jump step of $l$
		\begin{itemize}
			\item Delete $s_i$
			\item Replace all steps in $\rightViz(s_1)$ by horizontal steps
		\end{itemize}		 
\end{itemize}
These two maps are illustrated schematically below.
\begin{center}
		\includegraphics[scale=1.0]{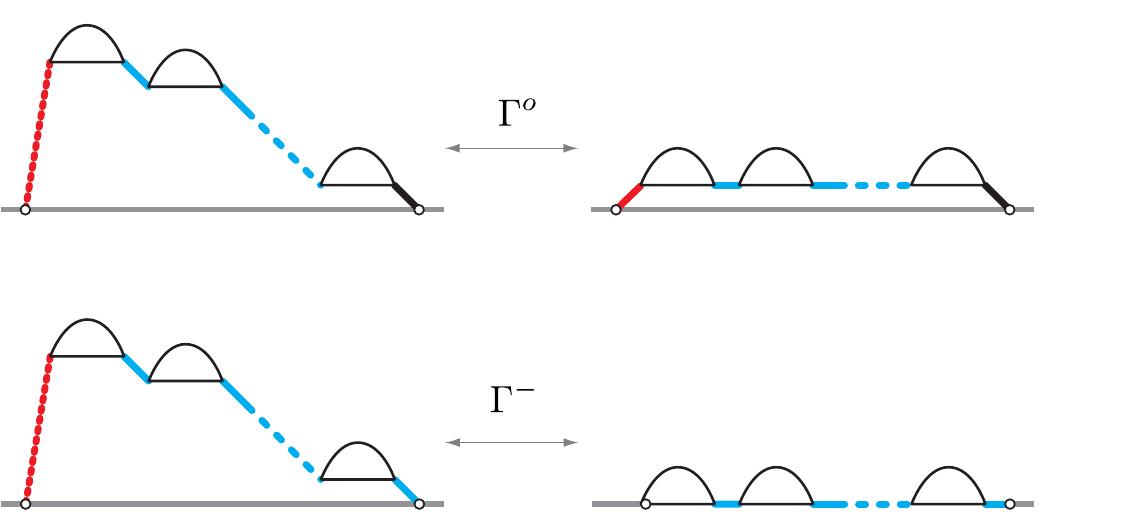}
\end{center}Note that 1) since there are no horizontal steps in any $l\in \lnp$ (as all jump $j$ steps have $j>0$) it is clear from \figref{fig_lukafact} that the defining characteristic of the  \includegraphics[scale=1.0]{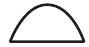} factor is that it contains no horizontal steps on its `surface' (a horizontal line at the same height of the left vertex of the first step of the factor), 2) $\gao$ does not change the number of steps (hence the superscript ``o'') whilst $\gam$ decreases the number of steps by one and 3) the height of the first vertex of the jump step $s_i$ and height of the last vertex in $\rightViz(s_i)$ is unchanged under the action of either map.

The bijection $\Ga_n: \lnp\to M_n $ is then defined as follows
\begin{itemize}
	\item If $l\in L_n$ apply $\gao$ to all jump steps $s_i\in l$.
	\item If $l\in L_{n+1}$ apply $\gam$ to the first (\emph{i.e.}\ leftmost) jump step and $\gao$ to all the remaining jump steps. 
\end{itemize}	
An example of the action of $\Ga_n$ is show below.
\begin{center}
			\includegraphics[width=\textwidth]{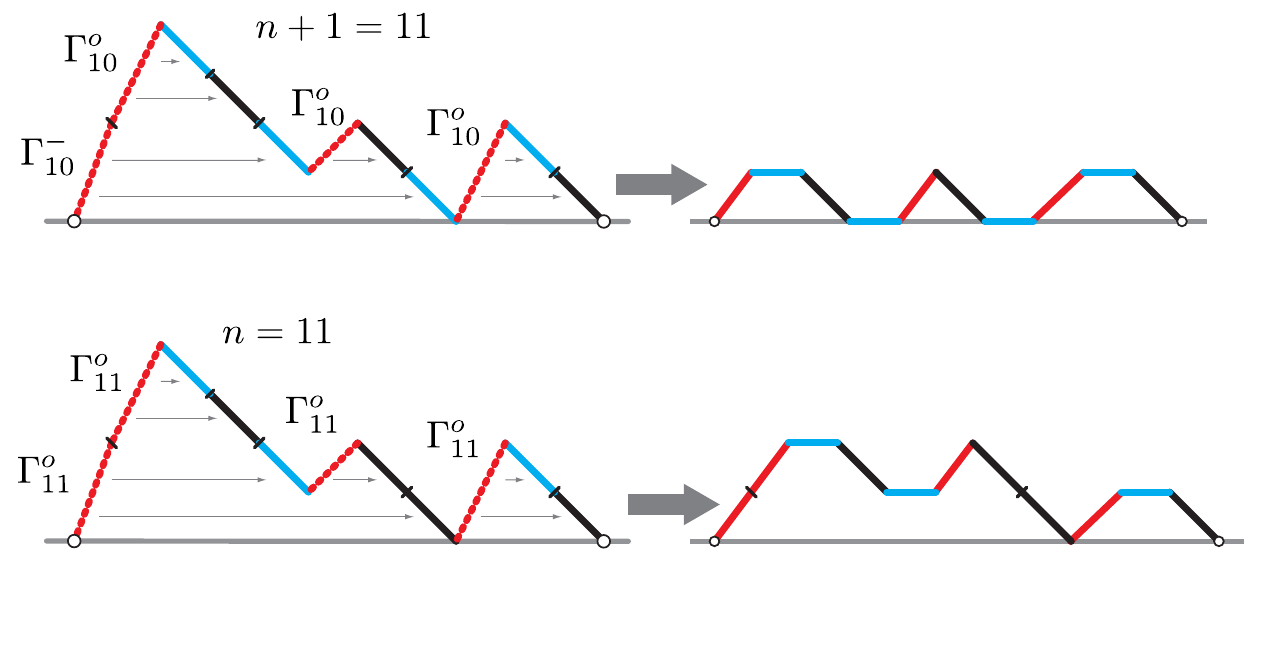}
\end{center}

\noindent \emph{Well defined:}	As noted above,  neither $\gao$ nor $\gam$ change the height of the   first and last vertices 
and hence their recursive action results in a  path which has all steps above the surface.  Since all steps are either up, 
down or horizontal the resulting path must be a Motzkin path. 
$\gao$ does not change the number of steps hence  if $l\in L_n$ then $\Ga_n(l)\in M_n$.  $\gam$ only acts on paths in $L_{n+1}$ and only acts on the first jump step thus if $l\in L_{n+1}$ then $\gam$ decreases the number of steps by one and hence $\Ga_n(l)\in M_n$. Thus $\Ga_n: \lnp\to M_n$.
\medskip

\noindent \emph{Injective:} We need to show that if $m=\Ga_n(l)$ and $m=\Ga_n(l')$ then $l=l'$. We show this by arguing that $m$ uniquely defines its preimage in $\lnp$ \emph{i.e.}\ $\Ga_n$ has an inverse. Note, there are no horizontal steps in any $l\in \lnp$ as all jump $j$ steps have $j>0$. All the horizontal steps of $m$ on the surface come from the action of $\Ga^-_n$ on an $l\in L_{n+1}$ path. Thus if $m$ has $j$ surface horizontal steps then the first jump step is uniquely jump $j$ and the surface horizontal steps become down steps. An up step  in $m$ arose from a unique jump step in $l\in \lnp$, the jump height of the jump step $j$ is uniquely determined by the number of horizontal steps on the ``surface'' (at the height of the up step) between the up step and the right-visible down step. If there are $k$ such horizontal steps then the up step becomes a $k+1$ jump step and the horizontal steps are replaced by down steps. Two examples are shown below for $\Ga_{13}$.
\begin{center}
 
			\includegraphics[scale=0.8]{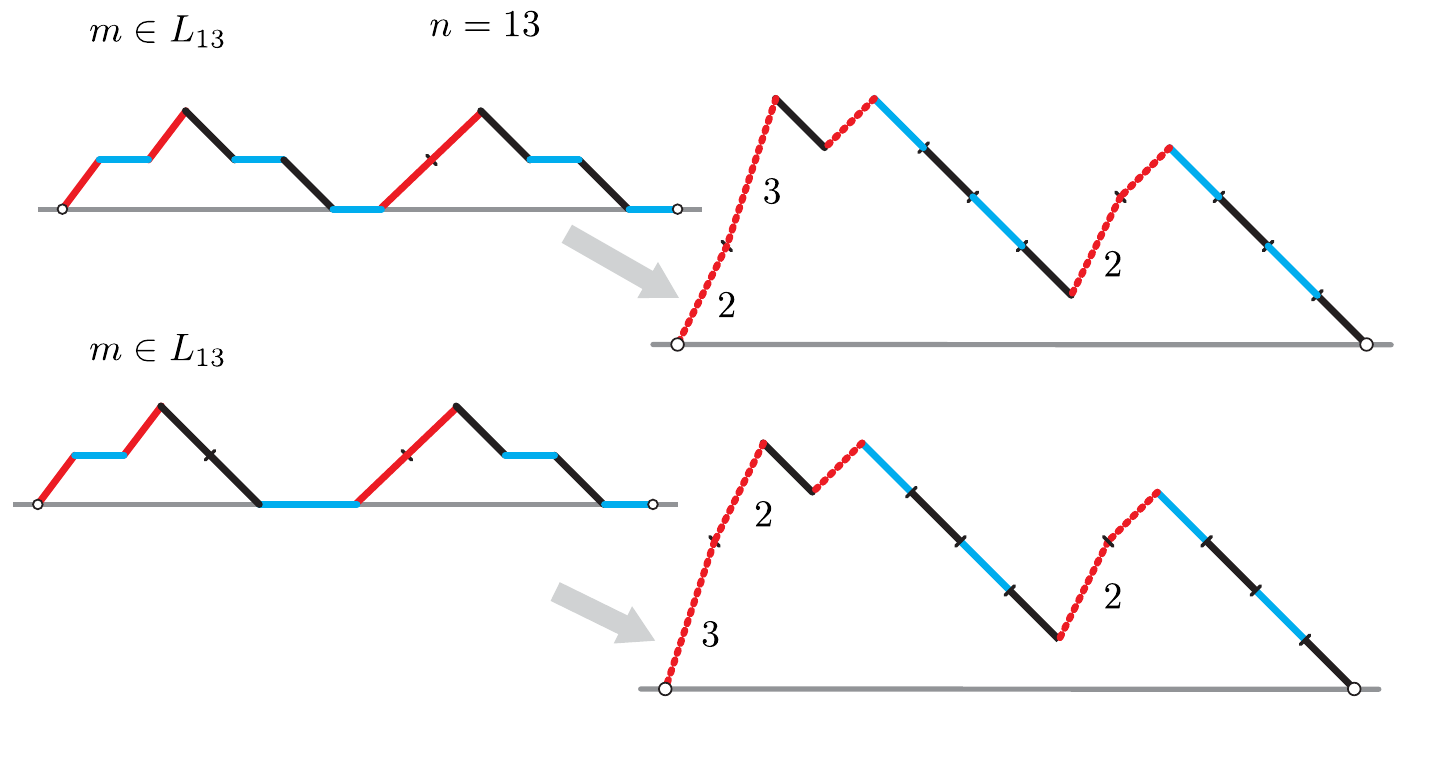}
\end{center}
The above example also illustrates how the order of two consecutive jump steps of different jump heights, in otherwise identical paths, arise from (or give rise to) different paths $m$.
\medskip

\noindent \emph{Surjective:} The injective paragraph  defines the inverse of $\Ga_n$ which clearly applies to every $m\in M_n$ and hence there exists an $l\in \lnp$ for which $m$ is an image \emph{i.e.}\ $\Ga$ is surjective. Those $m$ with any surface horizontal steps map to $L_{n+1}$ (as an initial jump step is added) and all the paths $m$ without horizontal surface steps map to $L_n$.  
\end{proof}

\section{Area under $(k,\ell)$-Lukasiewicz paths} 
\label{sec:lukasiewicz_paths_}

In this section we find a  $q$-deformed algebraic equation satisfied by the area-weighted generating function  of $(k,\ell)$-restricted Lukasiewicz paths which we solve in two cases, namely $(k,k)$ and $(0,\infty)$. The later solution can also be obtained via a bijection to Dyck paths which we present. 
 
\begin{definition}\label{def_area}
The \textbf{area} of a $(k,\ell)$-restricted Lukasiewicz path is the sum of the height of all
its vertices.
\end{definition}

Equivalently, one can connect all consecutive vertices of a path by straight lines and consider the
area enclosed by the path and the horizontal line that connects its end points.

Area-weighted Lukasiewicz paths satisfy a generalization of Theorem \ref{thm_lukaeqs}. 

\begin{theorem}\label{thm_lukaeqs_q}
Let $\Rgf{k}{\ell}(z;a,q)$ be the generating function for $(k,\ell)$-restricted Lukasiewicz paths keeping
track of contacts and area.  With respect to the partition functions, $\reslukapart{k}{\ell}_n(a,q)$, we have
\begin{equation}
	\Rgf{k}{\ell}(z;a,q)=\sum_{n\ge0} \reslukapart{k}{\ell}_n(a,q) z^n.
\end{equation}
The generating function $\Rgf{k}{\ell}(z;a,q)$ is given by the following pair of $q$-deformed algebraic equations
\begin{align}
	\Rgf{k}{\ell}(z;a,q)=1+ az \sum_{j=k}^\ell\left(\prod_{i=1}^jzq^i\Lgf{k}{\ell}(q^iz;q)\right) \Rgf{k}{\ell}(z;a,q)\;,\label{eq_rgf_q}\\
		\Lgf{k}{\ell}(z;q)=1+ \sum_{j=k}^{\ell}\left(\prod_{i=0}^jzq^i\Lgf{k}{\ell}(q^iz;q)\right)  \label{eq_lgf_q}
\end{align}
where $\Lgf{k}{\ell}(z;q)=\Rgf{k}{\ell}(z;1,q)$. 
\end{theorem}

Again, this is a standard generalisation of known methods. Note that in contrast to Theorem \ref{thm_lukaeqs}, a
Lukasiewicz path raised by  height $i$ leads to a term $q^i\Lgf{k}{\ell}(q^iz;q)$ due to the inclusion of
area weights.

An advantage of the inclusion of area weights is that we can express $\Rgf{k}{\ell}(z;a,q)$ via the solution
of a \emph{linear} $q$-difference equation. To obtain this result, we substitute
\begin{equation}
\Lgf{k}{\ell}(z;q)=\dfrac{H(qz;q)}{H(z;q)}
\end{equation}
into equation (\ref{eq_lgf_q}), and note that equation (\ref{simple2}) (solved for $R$) holds also when area weights are included. Thus we obtain the following theorem.

\begin{theorem}\label{thm_lukaeqs_q_linear}
\begin{equation}
	\Rgf{k}{\ell}(z;a,q)=\left(1-a+a\frac{\H{k}{\ell}(z;q)}{\H{k}{\ell}(qz;q)}\right)^{-1}
\end{equation}
where
\begin{align}
\H{k}{\ell}(qz;q)=\H{k}{\ell}(z;q)+\sum_{j=k}^\ell z^{j+1}q^{\binom{j+1}2}\H{k}{\ell}(q^{j+1}z;q)\;.
\end{align}
\end{theorem}

As $\Lgf{k}{\ell}(z;q)$ is a combinatorial generating function, it follows that $\H{k}{\ell}(z,q)$
must be of the form
\begin{equation}
\H{k}{\ell}(z;q)=\sum_{n=0}^\infty z^n\c{k}{\ell}_n(q)\;.
\end{equation}
This leads to the recurrence
\begin{equation}\label{recurrence}
(q^n-1)\c{k}{\ell}_n(q)=\sum_{j=k}^\ell q^{\binom{j+1}2+(j+1)(n-j-1)}\c{k}{\ell}_{n-j-1}(q)\;,
\end{equation}
with $\c{k}{\ell}_n(q)=0$ for $n<0$. As $\H{k}{\ell}(z;q)$ is determined up to a multiplicative constant, 
we let $\c{k}{\ell}_0(q)=1$ without loss of generality.

We can solve the recurrence in (\ref{recurrence}) explicitly for $(k,k)$-restricted and $(0,\infty)$-restricted
Lukasiewicz paths. The resulting $q$-series are summarised as follows.

\begin{theorem}\label{thm:qseries}
\begin{align}
\H{k}{k}(z,q)=&\sum_{n=0}^\infty\frac{q^{\binom{(k+1)n}2}(-z^{k+1})^n}{(q^{k+1};q^{k+1})_n}
\intertext{and}
\H{0}{\infty}(z,q)=&\sum_{n=0}^\infty\frac{q^{n^2-n}(-z)^n}{(q;q)_n}\;,
\end{align}
with the $q$-product notation
\begin{equation}
(t;q)_n=\prod_{j=0}^{n-1}(1-tq^j)\;.
\end{equation}
\end{theorem}

Note that this even makes sense for $k=0$, where Euler's product formula implies that
\begin{equation}
\H{0}{0}(z,q)=\sum_{n=0}^\infty\frac{q^{\binom{n}2}(-z)^{n}}{(q;q)_n}=(z;q)_\infty
\end{equation}
and hence
\begin{equation}
\Rgf{0}{0}(z;a,q)=\left(1-a+a(z;q)_\infty/(qz;q)_\infty\right)^{-1}=\frac1{1-az}
\end{equation}
as trivially required.

More importantly, note that
\begin{equation}
\H{0}{\infty}(z,q)=\H{1}{1}(qz^2,q^2)
\end{equation}
which implies that there must be a bijection between $(0,\infty)$-restricted Lukaziewicz paths
and Dyck paths counted by length, contacts, and area. More precisely, this observation provides a 
generating function proof of the following theorem.

\begin{theorem}\label{thm:areabij}
There exists a bijection between $(0,\infty)$-restricted Lukaziewicz paths of length $n$ and 
area $m$ and Dyck paths of length $2n$ and area $2m+n$, which preserves the number of contacts.
\end{theorem}

We now give an explicit bijective proof of this theorem. Recall that Dyck paths have an even number of steps, and the
difference of their vertex coordinates is even. Dyck paths are uniquely determined by their down steps.
 
Given a Dyck path $w_0w_1\ldots w_{2n}$, for each integer $j$ with $1\leq j\leq n$ there is a unique down-step 
starting at a vertex $w_{i_j}=(i_j,h_{i_j})$ with coordinate difference ${i_j}-h_{i_j}=2j$. This 
down-step gets mapped to $v_j=(j,h_{i_j}-1)$. Necessarily $h_{i_j}\geq1$, and the height between two subsequent
down-steps cannot decrease by more than one, i.e. $h_{i_{j+1}}-h_{i_j}\geq-1$. As there is a down-step starting
at $(2n-1,1)$, necessarily $v_n=(n,0)$. If we define $v_0=(0,0)$, the 
resulting path $v_0v_1\ldots v_n$ is therefore a Lukasiewicz path.

Conversely, given a Lukasiewicz path $v_0v_1\ldots v_n$, for each integer $i$ with $1\leq i\leq n$ we map the 
vertex $v_i=(i,h_i)$ to a down-step starting at the vertex $w_{2i+h_i-1}=(2i+h_i-1,h_i+1)$ (and the associated
up-step, which is defined implicitly). Subsequent down-steps are separated by precisely $h_{i+1}-h_i+1$ up-steps. 
$v_n=(n,0)$ implies that $w_{2n}=(2n,0)$, and if we define $w_0=(0,0)$ and augment with the intermediate 
up-steps, the resulting path $w_0w_1\ldots w_{2n}$ is a Dyck path.

Clearly both mappings are injective, and therefore also bijective. One can easily check that they are inverses
of each other.

There is a one-to-one mapping between vertices of height zero, and as a vertex with height $h$ of a Lukasiewicz 
path gets mapped to a pair of vertices of a Dyck path at height $h$ and $h+1$ (the starting vertices of an
associated pair of up/down-steps). Therefore, a Dyck path of area $m$ gets mapped to a Lukasiewicz path of area
$2m+n$.

Alternatively, this mapping can be visualised by considering the figure below. 
\begin{center}
		\includegraphics[scale=0.7]{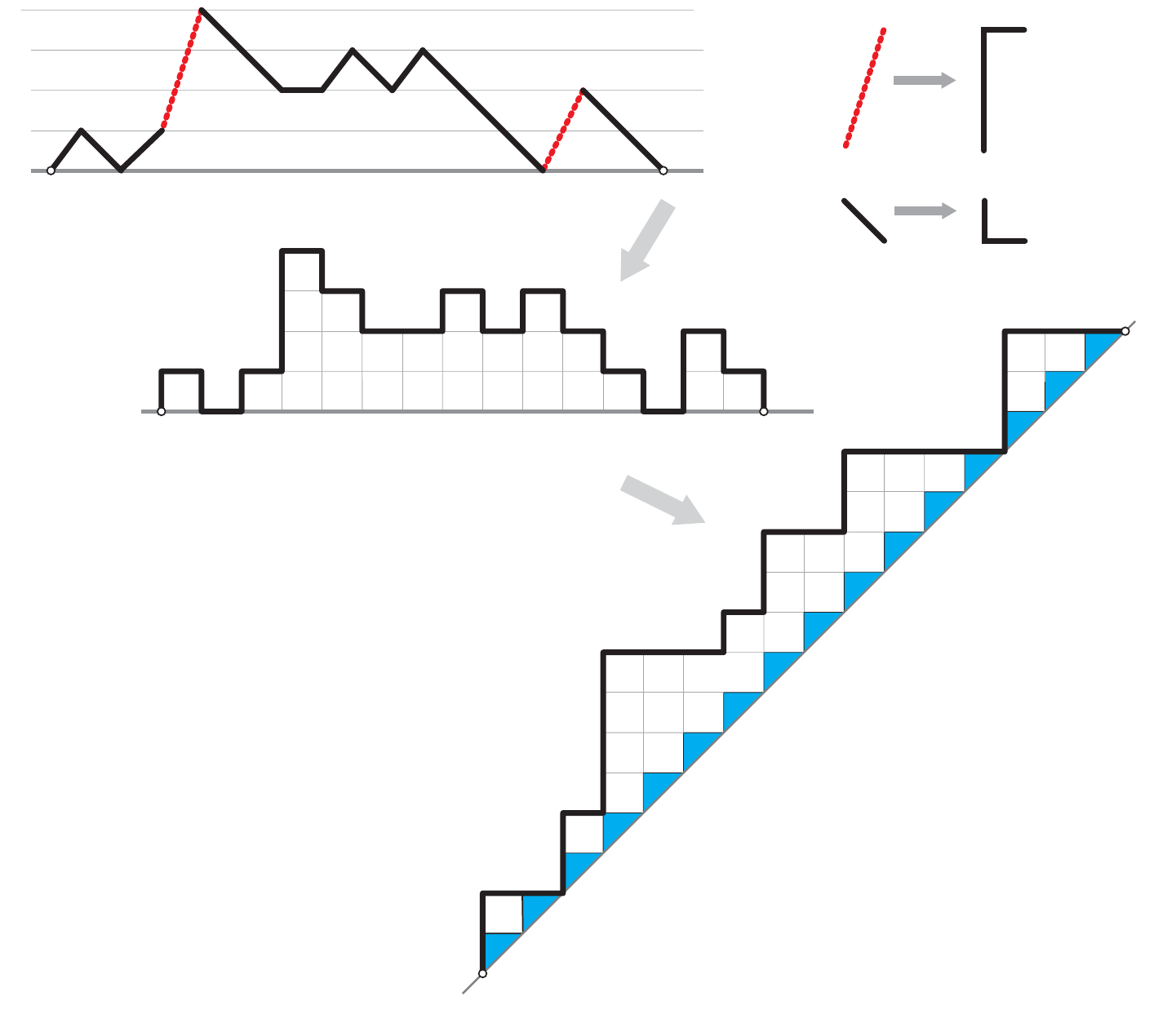}
\end{center}

Due to the bijection to Dyck paths, the associated phase diagram can be found in  \cite{Owczarek:2010lr}. 
For $0<q<1$ paths are bound to the surface, while for $q>1$ configurations with maximal area dominate the
ensemble irrespective of the value of $a$. Only for $q=1$ does there exist a genuine binding/unbinding transition
when varying $a$.

\section*{Acknowledgements} 

Financial support from the Australian Research Council  and The Centre of Excellence for the Mathematics and Statistics of Complex Systems (\mbox{MASCOS}) is gratefully
acknowledged.

\end{document}